\documentclass[12pt]{article}
\usepackage{amssymb}
\usepackage{amsmath,amsthm}
\usepackage[latin1]{inputenc}
\usepackage{hyperref}
\usepackage{color}
\usepackage{graphicx}
\hypersetup{colorlinks=true, linkcolor=blue, citecolor=blue,
urlcolor=blue}

 \newtheorem{remark}{Remark}

 \newtheorem{lemma}[remark]{Lemma}
 \newtheorem{theorem}[remark]{Theorem}
 \newtheorem{proposition}[remark]{Proposition}
 \newtheorem{corollary}[remark]{Corollary}

\title{Coloring, location and domination of corona graphs}

\author{I. Gonz\'alez Yero$^{\S}$, D. Kuziak$^{\ddag}$ and A. Rond\'on Aguilar$^{\dag}$\\
\\
$^{\S}${\small Departamento de Matem\'aticas, Escuela Polit\'ecnica Superior de Algeciras}\\
{\small Universidad de C\'adiz,} {\small
Av. Ram\'on Puyol, s/n, 11202 Algeciras, Spain.} \\ {\small
ismael.gonzalez\@@uca.es}\\
$^{\ddag}${\small Faculty of Applied Physics and  Mathematics}\\
{\small Gda\'nsk University of Technology,} {\small ul. Narutowicza
11/12 80-233 Gda\'nsk, Poland.} \\ {\small dkuziak\@@mif.pg.gda.pl}\\
{\small $^{\dag}$Filial Universitaria Municipal Media Luna, Universidad de Granma,}\\
{\small Carretera de Manzanillo Km. 1, 87700 Media Luna, Granma, Cuba.} \\
{\small arondona\@@udg.co.cu}}

\begin{document}

\maketitle

\begin{abstract}
A vertex coloring of a graph $G$ is an assignment of colors to
the vertices of  $G$ such that every two adjacent vertices of
$G$ have different colors. A coloring related property of a graphs is also an assignment of colors or labels to the vertices of a graph, in which the process of labeling is done according to an extra condition. A set $S$ of vertices of a graph $G$ is a dominating set in $G$ if
every vertex outside of $S$ is adjacent to at least one vertex
belonging to $S$. A domination parameter of $G$ is related to those
structures of a graph satisfying some domination property together
with other conditions on the vertices of $G$. In this article we study several mathematical properties related to coloring, domination and location of corona graphs.

We investigate the distance-$k$ colorings of corona graphs.
Particularly, we obtain  tight bounds for the distance-$2$ chromatic
number and distance-$3$ chromatic number of corona graphs, throughout some
relationships between the distance-$k$ chromatic number of corona
graphs and the distance-$k$ chromatic number of its factors.
Moreover, we give the exact value of the distance-$k$ chromatic
number of the corona of a path and an arbitrary graph. On the other hand, we
obtain bounds for the Roman dominating number and the locating-domination number of corona graphs. We give
closed formulaes for the $k$-domination number, the distance-$k$
domination number, the independence domination number, the domatic
number and the idomatic number of corona graphs.
\end{abstract}



\section{Introduction}

Nowadays the studies about the behavior of several graph parameters
in product graphs have become into an interesting topic of research
in graph theory. For instance, is it well known the Hedetniemi's
coloring conjecture \cite{hedet-conj-first,hedet-conjec-survey} for
the categorical product (or direct product), which states that the chromatic number of
categorial product graphs is equal to the minimum value between the
chromatic numbers of its factors. Also, one of the oldest open
problems in domination in graphs is related with product graphs. The
problem was presented first by Vizing \cite{vizing1} in 1963. After
that he pointed out as a conjecture in \cite{vizing}. The conjecture
states that the domination number of Cartesian product graphs is
greater than or equal to the product of the domination numbers of
its factors.

A graph labeling is an assignment of labels, traditionally
represented by integers or  colors, to the edges or vertices, or
both, of a graph. Formally, given a graph $G$, a vertex labeling is
a function mapping vertices of $G$ to a set of labels. One of the
most popular graphs labeling is the graph coloring, which is an
assignment of colors to the vertices or edges, or both, of a graph.
For instance, given a set of colors $C=\{c_1,c_2,...,c_r\}$, a
vertex coloring of a graph $G=(V,E)$ is a map $c: V \rightarrow C$
such that for every two adjacent vertices $u,v\in V$ it follows
$c(u)\ne c(v)$. The minimum value $r=|C|$ for which $G$ has a vertex
coloring is called the {\em chromatic number} of $G$ and it is
denoted by $\chi(G)$. Nowadays, there are several kinds of
investigations related to vertex colorings of graphs (for example
\cite{bollobas,chartrand,lawler}).

Coloring problems in graphs have been related to several number of
scheduling problems \cite{color-aplic-2}. For instance, the
scheduling problem of assigning aircrafts to flights, the assignments
of tasks to time slots or assigning frequency channels to different
wireless applications \cite{dist-k-color-3}. Moreover, graph
colorings can be applied to register allocation
\cite{color-aplic-1}, pattern matching or some recreational games
like the well known puzzles called Sudoku. On the other hand, the
chromatic number has been related with several parameters of
graphs, and as a consequence, there exists now different types of vertex colorings such as
list coloring, total coloring, acyclic coloring, distance-$k$
coloring, etc.

A set $S$ of vertices of a graph $G$ is an {\em independent set} of
$G$ if for every $v\in S$ it is satisfied that $\delta_S(v)=0$. The
minimum cardinality of any independent set in $G$ is called the {\em
independence number} and it is denoted by $\beta_0(G)$. Also, a set
$S$ is a $t$-{\em dependent set} in $G$, if for every vertex $v\in
S$ it follows that $\delta_{S}(v)\le k$. Similarly, the minimum
cardinality of any $k$-dependent set in $G$ is the $k$-{\em
dependence number} and it is denoted by $\beta_k(G)$.

The set of vertices $D\subset V$ is a {\em dominating set} if for
every vertex $v\in \overline{D}$ it is satisfied that
$\delta_D(v)\ge 1$ \cite{bookdom1}. The minimum cardinality of any
dominating of $G$ is the {\em domination number} of $G$ and it is
denoted by $\gamma(G)$. Moreover, the set $D$ is $k$-{\em
dominating}, $k\ge 2$, if for every vertex $v\in \overline{D}$, it
follows $\delta_S(v)\ge k$. The minimum cardinality of any
$k$-dominating set in $G$ is the $k$-{\em domination number} and it
is denoted by $\gamma_k(G)$. The concept of domination has been
related with several structures of the graph, which has led to
different kind of domination parameters associated to some extra
conditions. In this sense, some of the most popular cases are the
independent dominating sets, connected dominating sets, convex
dominating sets, distance-$k$ dominating sets, domatic partitions,
etc. For general notation and terminology in domination we follow the
books \cite{bookdom1,bookdom2}.

We begin by establishing the principal terminology and notation
which we will use throughout the article. Hereafter $G=(V,E)$
represents a undirected finite graph without loops and multiple
edges with set of vertices $V$ and set of edges $E$. The order of
$G$ is $|V|=n(G)$ and the size $|E|=m(G)$ (If there is no ambiguity
we will use only $n$ and $m$). We denote two adjacent vertices
$u,v\in V$ by $u\sim v$ and in this case we say that $uv$ is an edge
of $G$ or $uv\in E$. For a nonempty set $X\subseteq V$, and a vertex
$v\in V$, $N_X(v)$ denotes the set of neighbors that $v$ has in $X$:
$N_X(v):=\{u\in X: u\sim v\}$ and the degree of $v$ in $X$ is
denoted by $\delta_{X}(v)=|N_{X}(v)|.$ In the case $X=V$ we will use
only $N(v)$, which is also called the open neighborhood of a vertex
$v\in V$, and $\delta(v)$ to denote the degree of $v$ in $G$. The
close neighborhood of a vertex $v\in V$ is $N[v]=N(v)\cup \{v\}$.
The minimum and maximum degrees of $G$ are denoted by $\delta$ and
$\Delta$, respectively. The subgraph induced by $S\subset V$ is
denoted by $\langle S\rangle $ and the complement of the set $S$ in
$V$ is denoted by $\overline{S}$. The distance
between two vertices $u,v\in V$ of $G$ is denoted  by $d_G(u,v)$ (or
$d(u,v)$ if there is no ambiguity). The diameter of a graph is the
maximum of the distances between any two vertices of $G$ and it is
denoted by $D(G)$. Given a vertex $v$ of $G$, we denote by $M_t[v]$ the set of vertices of $G$ whose distance to the vertex $v$ is less than or equal to $t$, {\em i.e.} $M_t[v]=\{u\in V\;:\;d(u,v)\le t\}$.  Throughout the article, given
the set of colors $C$, we will refer to the map $c: V\rightarrow C$
as a distance-$k$ coloring of vertices of $G$.

The corona product graph (corona graph, for short) of two graphs
was introduced first by Frucht and Harary in \cite{corona-first}.
After that many works have been devoted to study its structure and
to obtain some relationships between the corona graph and its
factors \cite{corona-spectrum,corona-banwidth,corona-first,corona-profile,corona-sum}.

Let $G$ and $H$ be two graphs of order $n_1$ and $n_2$,
respectively. The corona graph $G\odot H$ is defined as the graph obtained
from $G$ and $H$ by taking one copy of $G$ and $n_1$ copies of $H$
and joining by an edge each vertex from the $i^{th}$-copy of $H$
with the $i^{th}$-vertex of $G$. Hereafter, we will denote by
$V=\{v_1,v_2,...,v_n\}$   the set of vertices of $G$ and by
$H_i=(V_i,E_i)$  the $i^{th}$ copy of $H$ in $G\odot H$.

\section{Distance-$k$ coloring}

A distance-$k$ coloring of a graph $G$ is an assignment of colors to
the vertices of  $G$ such that every two different vertices $a,b$ of
$G$ have different colors if the distance between $a$ and $b$ is
less than or equal to $k$
\cite{dist-k-color-1,dist-k-color-2,dist-k-color-3}. The minimum
number of colors in a distance-$k$ coloring of $G$ is the {\em
distance-$k$ chromatic number} of $G$ and it is denoted by
$\chi_{\le k}(G)$. Notice that the case $k=1$ corresponds to the
standard well known vertex coloring.

We begin this section by presenting the following almost
straightforward result relative to the distance-$1$ chromatic number
({\em i.e.} the standard chromatic number) of any corona graph
$G\odot H$ and further we will analyze the distance-$k$ chromatic
number of $G\odot H$, with $2\le k\le D(G\odot H)$.

\begin{remark}
For any graphs $G$ and $H$,
$$\chi(G\odot H)=\max\{\chi(G),\chi(H)+1\}.$$
\end{remark}

\begin{proof}
Since every vertex $u\in V_i$ is adjacent to $v_i\in V$ we obtain
that $\chi(G\odot H)\ge \chi(H)+1$.  Also, it is clear that
$\chi(G\odot H)\ge \chi(G)$.

On the other hand, let $t=\max\{\chi(G),\chi(H)+1\}$ and let us
color the vertices $v_i$, $i\in \{1,...,n_1\}$,  of $G$ by using $t$
different colors $c_1,c_2,...,c_t$. Now, if $c(v_i)=c_i$, then the
copy $H_i$ of $H$ can be colored by using the set of colors
$c_1,c_2,...,c_{i-1},c_{i+1},...,c_t$. Therefore, we obtain that
$\chi(G\odot $H$)\le t=\max\{\chi(G),\chi(H)+1\}$ and the result
follows.
\end{proof}

\begin{remark}
Let $G$ be a graph of order $n$ and let $k\ge 1$ be an integer. Then $\chi_{\le k}(G)=n$ if and
only if $D(G)\le k$.
\end{remark}

\begin{proof}
If $D(G)\le k$, then for every two different vertices $u,v$  of $G$
we have that $d(u,v)\le D(G)\le k$. So, $u$ and $v$ have different colors
in any distance-$k$ coloring of $G$ and, as a consequence,
$\chi_{\le k}(G)=n$.

On the other side, let $G$ be a graph such that $\chi_{\le k}(G)=n$.
Let us suppose $D(G)>k$ and let $a,b$ be two different vertices of
$G$ such that $d(a,b)=D(G)$. Let $C=\{c_1,c_2,...,c_{n-1}\}$ be a
set of $n-1$ colors. Hence, we can color the set of vertices of
$G-\{b\}$ with the $n-1$ colors in $C$. Now, as $k<D(G)=d(a,b)$ we
can color vertex $b$ by using the color of vertex $a$. Thus
$G$ can be colored with $n-1$ colors, which is a contradiction
Therefore, we have that $D(G)\le k$.
\end{proof}

As a consequence of the above remark, from now on we will
focus on the cases $2\le k\le D(G)-1$. Also, as for every
graph $G$ and $H$ we have that $D(G\odot H)=D(G)+2$ we are
interested here in obtaining the distance-$k$ chromatic number of
corona graphs $G\odot H$ for $2\le k\le D(G)+1$.

\begin{theorem}\label{dist-2-coloring}
Let $G$ be a graph of maximum degree $\Delta_1$ and let $H$ be a graph of order $n_2$. Then,
$$\Delta_1+n_2+1\le \chi_{\le 2}(G\odot H)\le \chi_{\le 2}(G)+n_2.$$
\end{theorem}

\begin{proof}
The lower bound is a direct consequence of Theorem \ref{cota-k-par} (See Appendix)
by taking into account that the maximum degree of $G\odot H$ is
$n_2+\Delta_1$.

On the other hand, let $t=\chi_{\le 2}(G)+n_2$ and let $C=\{c_1,c_2,...,c_t\}$ be a set of pairwise distinct colors.  Let us color the vertices of $G$ by
using $\chi_{\le 2}(G)$ different colors and let us suppose that for
$v_i\in V$ we have $c(v_i)=c_i$. Hence, for every distinct vertices
$v_j,v_l\in N_V[v_i]$ we have $c(v_j)\ne c(v_l)$. Also, as
$\chi_{\le 2}(G)\ge\delta_{G}(v_i)+1$ for every $v_i\in V$ we obtain
that
$$n_2=t-\chi_{\le 2}(G)\le t-(\delta_{G}(v_i)+1).$$
Now, let $C_i=\{c_{i_1},c_{i_2},...,c_{i_r}\}$ be such that for every
$c_{i_j}\in C_i$ there exists  $v_j\in N_V[v_i]$ with
$c(v_i)=c_{i_j}$. Since $|C-C_i|=t-(\delta_{G}(v_i)+1)\ge n_2$ we
obtain that the vertices of the copy $H_i$ of $H$ can be colored
with the colors in $C-C_i$. Therefore, $\chi_{\le 2}(G\odot
H)\le t=\chi_{\le 2}(G)+n_2$.
\end{proof}

The following corollary shows that the above bounds are tight.

\begin{corollary} Let $H$ be any graph of order $n_2$. Then,
\begin{itemize}
\item[{\rm(i)}] If $n_1\ge 3$, then $\chi_{\le 2}(P_{n_1}\odot H)=n_2+3$.
\item[{\rm(ii)}] For any positive integer $t$, $\chi_{\le 2}(C_{3t}\odot H)=n_2+3$.
\item[{\rm(iii)}] For any tree $T$ of maximum degree $\Delta_1$, $\chi_{\le 2}(T\odot H)=n_2+\Delta_1+1$.
\end{itemize}
\end{corollary}

Notice that, for instance, if $G$ is a cycle of order $3t+1$ or
$3t+2$, with $t\ge 1$ an integer,  then $\chi_{\le
2}(C_{3t+1})=4$ and $\chi_{\le 2}(C_{3t+2})=5$. Thus, we have that
$5=\chi_{\le 2}(C_{3t+1}\odot N_2)<6$ and $5=\chi_{\le
2}(C_{3t+2}\odot N_2)<6$. Also, for the complete bipartite graph
$K_{s,t}$, with $2<s\le t$, we have that $\chi_{\le 2}(K_{s,t})=s+t$
and $t+2<\chi_{\le 2}(K_{s,t}\odot K_1)=s+t<s+t+1$.

\begin{theorem}
Let $G$ be a graph of minimum and maximum degree $\delta_1$ and $\Delta_1$, respectively and let $H$ be a graph of order $n_2$. Then
$$\chi_{\le 3}(G\odot H)\le \chi_{\le 3}(G)+n_2(\Delta_1+1).$$
Moreover, if $G$ is triangle free, then
$$\chi_{\le 3}(G\odot H)\ge 2n_2+\Delta_1+\delta_1.$$
\end{theorem}

\begin{proof}
Let $t=\chi_{\le 3}(G)+n_2(\Delta_1+1)$ and let $C=\{c_1,c_2,...,c_t\}$ be a set of pairwise distinct colors.  Let us color the vertices of $G$ by using
$\chi_{\le 3}(G)$ different colors and let us suppose that for
$v_i\in V$ we have that $c(v_i)=c_i$.

Now, let $C_i=\{c_{i_1},c_{i_2},...,c_{i_r}\}$ be such that for every
$c_{i_j}\in C_i$ there exists  $v_j\in \displaystyle\bigcup_{v_l\in
N_V[v_i]}N_V[v_l]$ with $c(v_j)=c_{i_j}$. Thus, $|C_i|\le \chi_{\le
3}(G)$ and we have that

$$|C-C_i|= t-|C_i|\ge t-\chi_{\le 3}(G)=n_2(\Delta_1+1).$$
So, we obtain that the vertices of the $\Delta+1$ copies of $H$
corresponding to the vertices  of $G$ in $N_V[v_i]$ can be colored
with the colors in $C-C_i$. Also, if $v\in V_l$ such that
$v_l\notin N_V[v_i]$, then there exists a vertex $u\in V_r$, with
$v_r\in N_V[v_i]$, such that $d_{G\odot H}(u,v)>3$. Thus, $v$ can be
colored by using one color from the set of colors in $C-C_i$.
Therefore, $\chi_{\le 3}(G\odot H)\le t=\chi_{\le
3}(G)+n_2(\Delta_1+1)$.

On the other hand, let us suppose $G$ is triangle free. Let $v_i\in
V$ be a vertex of maximum degree in $G$ and let $v_j\in N_V[v_i]$,
$j\ne i$. Hence, for every two different vertices $u,v\in (V_i\cup
V_j)\cup(N_V[v_i]\cup N_V[v_j])$ we have $d_{G\odot H}(u,v)\le 3$.
So, we have that $c(u)\ne c(v)$. Thus, we obtain that
$$\chi_{\le 3}(G\odot H)\ge |(V_i\cup V_j)\cup(N_V[v_i]\cup N_V[v_j])|\ge 2n_2+\Delta_1+\delta_1.$$

\end{proof}

Notice that the above bounds are tight. For instance, if $H$ is a
graph of order $n_2$, then  the lower bound is achieved for the case
of $C_4\odot H$, where we have $\chi_{\le 3}(C_4\odot H)=2n_2+4$.
Moreover, the upper bound is tight for the corona graph
$K_{n_1}\odot H$, in which case it is satisfied that $\chi_{\le
3}(K_{n_1}\odot H)=n_1n_2+n_1$. Next we study the distance-$k$
chromatic number of some particular cases of corona graphs.

\begin{proposition}
Let $H$ be a graph of order $n_2$ and let $T=(V,E)$ be a tree. Let
$v_i,v_j\in V$ such that
$\Delta_{ij}(T)=\delta(v_i)+\delta(v_j)=\max\{\delta(v_l)+\delta(v_r)\;:\;v_l,v_r\in
V,\;v_l\sim v_r\}$. Then
$$\chi_{\le 3}(T\odot H)=2n_2+\Delta_{ij}(T).$$
\end{proposition}

\begin{proof}
Let $B=V_i\cup V_j\cup N_V[v_i]\cup N_V[v_j]$ be the set of vertices of
$T\odot H$. Since $T$ is a tree we  have that
$(N_V[v_i]-\{v_j\})\cap (N_V[v_j]-\{v_i\})=\emptyset$. Thus,
$|B|=|V_i|+|V_j|+|N_V[v_i]|+|N_V[v_j]|-2$. Also, for every two
different vertices $a,b\in B$ we have that $d_{T\odot H}(a,b)\le 3$
and, as a consequence, we obtain that $c(a)\ne c(b)$. Therefore,
\begin{align*}
\chi_{\le 3}(T\odot H)&\ge |B|\\
&=|V_i|+|V_j|+|N_V[v_i]|+|N_V[v_j]|-2\\
&=2n_2+\delta(v_i)+\delta(v_j)\\
&=2n_2+\Delta_{ij}(T).
\end{align*}
On the other hand, let $t=2n_2+\Delta_{ij}(T)$ and let $C=\{c_1,c_2,...,c_t\}$ be a set of pairwise distinct colors. Let us color the set of vertices of
$T$ by using the minimum number of colors from the set $C$ and let
us suppose that $c(v_i)=c_i$ and $c(v_j)=c_j$. Now, let
$C_{ij}=\{c_{ij_1},c_{ij_2},...c_{ij_{r}}\}\subset C$ be such that for
every $c_{ij_{l}}\in C_{ij}$ there exists $a\in N_V[v_i]\cup
N_V[v_j]$ with $c(a)=c_{ij_{l}}$. Notice that
$|C_{ij}|=\Delta_{ij}(T)$ and also, any vertex belonging to $V_i\cup
V_j$ can be colored by using the colors in $C-C_{ij}$.

Now, if $v_q\notin N_V[v_i]\cup N_V[v_j]$ there exists a vertex $a\in
N_V[v_i]\cup N_V[v_j]$  such that $d_{T\odot H}(a,v_q)>3$. Thus,
$v_q$ can be colored by using one of the colors in $C_{ij}$. Also,
if $b\in V_f$, with $f\ne i,j$, then there exist a vertex $b'\in
V_i\cup V_j$ such that $d_{T\odot H}(b',b)>3$. Thus, $b'$ can be
colored by using one color from the set of colors in $C-C'$.
Therefore, we have that $T\odot H$ can be colored with $t$ colors.
As a consequence, $\chi_{\le 3}(T\odot H)\le t=2n_2+\Delta_{ij}(T)$
and the result follows.
\end{proof}

\begin{proposition}
Let $H$ be a graph of order $n_2$ and let $n_1\ge 2$. Then for every $2\le k\le n_1$,
$$\chi_{\le k}(P_{n_1}\odot H)=\left\{\begin{array}{ll}
                                        n_2(k-1)+k+1, & {\rm if}\;\; k\le n_1-1,\\
                                        n_2(k-1)+k, & {\rm if}\;\; k= n_1.
                                      \end{array}
\right.$$
\end{proposition}

\begin{proof}
Let us suppose $k\le n_1-1$ and let $P_{n_1}=v_1v_2...v_{n_1}$ in
$P_{n_1}\odot H$. Hence,  there exists a vertex $v_i$ of degree two
in $P_{n_1}$ such that $v_{i+k-2}$ has degree two and for every two
different vertices $a,b\in
A=\{v_{i-1},v_i,v_{i+1},....,v_{i+k-2},v_{i+k-1}\}$ we have that
$d_{P_{n_1}\odot H}(a,b)\le k$. Thus, $c(a)\ne c(b)$. Now, let
$B=(\bigcup_{j=i}^{i+k-2}V_j)\cup A$. Hence, for every $a,b\in B$ we
have that $d_{P_{n_1}\odot H}(a,b)\le k$. Thus, $c(a)\ne c(b)$.
Therefore,
$$\chi_{\le k}(P_{n_1}\odot H)\ge |B|=|A|+\left|\bigcup_{j=i}^{i+k-2}V_j\right|=k+1+n_2(k-1).$$
On the other hand, let $t=k+1+n_2(k-1)$ and let
$C=\{c_1,c_2,...,c_t\}$ be a set of pairwise distinct colors. Now,  let
$Q=\{v_i,v_{i+1},...,v_{i+k-1},v_{i+k}\}$ be any $k+1$ consecutive
vertices in $P_{n_1}$. Hence, the vertices in $Q$ can be colored by
using $k+1$ different colors of $C$. Now, let $C'\subset C$ be such
that for every $c_l\in C'$ there exists $v\in Q$ with $c(v)=c_l$.
Hence, since $|C-C'|=n_1(k-1)$ we obtain that the vertices in
$\bigcup_{j=1}^{k-1}V_{i+j}$ can be colored by using the
colors in $C-C'$. Now, if $v_r\notin Q$, then there exists a vertex
$v_{q}\in Q$ such that $d_{P_{n_1}\odot H}(v_r,v_q)>k$. So, $v_r$
can be colored by using a color from the set $C'$. Also, if $u\in
V_f$ such that $v_f\notin Q-\{v_i,v_{i+k}\}$, then there exists a
vertex $v\in V_y$ such that $v_y\in Q-\{v_i,v_{i+k}\}$ for which
$d_{P_{n_1}\odot H}(u,v)>k$. So, $u$ can be colored by using a color
from the set $C-C'$. Therefore, $\chi_{\le t}(P_{n_1}\odot H)\le
t=k+1+n_2(k-1)$ and the result follows.

Now, let us suppose that $k=n_1$. So, for every different vertices
$v_i,v_j$ of $P_{n_1}$ in $P_{n_1}\odot H$  we have that $c(v_i)\ne
c(v_j)$. Now, let $B=(\bigcup_{i=1}^{n_1-1}V_i)\cup V$. Hence, we
have that, for every two different vertices $a,b\in B$,
$d_{P_{n_1}\odot H}(a,b)\le k$. Thus, $c(a)\ne c(b)$ and, as a
consequence,
$$\chi_{\le k}(P_{n_1}\odot H)\ge |B|=n_1+\left|\bigcup_{i=1}^{n_1-1}V_i\right|=n_1+n_2(n_1-1)=n_2(k-1)+k.$$
On the other hand, let $t=n_2(k-1)+k$ and let
$C=\{c_1,c_2,...,c_t\}$. Since $k=n_1$ we can color the set of
vertices of $P_{n_1}$ by using $k$ colors. Let $C'$ be the set of
colors used to color the set $V$. Since the distance between
the vertices in $V_1$ and $V_{n_1}$ is greater than $k$, these sets
can be assigned the same colors and the rest of the vertices in
$\bigcup_{i=1}^{n_1-1}V_i$ can be colored by using the colors in
$C-C'$. Therefore, we obtain that $\chi_{\le k}(P_{n_1}\odot H)\le
t=n_2(k-1)+k$ and the result follows.
\end{proof}

\section{Roman domination}

The concepts about Roman domination were introduced first by Steward
in \cite{roman-1} and studied further by some authors, for instance we mention \cite{roman}. A map $f : V
\rightarrow \{0, 1, 2\}$ is a {\em Roman dominating function} for a
graph $G$ if for every vertex $v$ with $f(v) = 0$, there exists a
vertex $u\in N(v)$ such that $f(u) = 2$. The {\em weight} of a Roman
dominating function is given by $f(V) =\sum_{u\in V}f(u)$. The
minimum weight of a Roman dominating function on $G$ is called the
{\em Roman domination number} of $G$ and it is denoted by
$\gamma_R(G)$. In this section we study the Roman domination number
of corona graphs.

Let $f$ be a Roman dominating function on $G$ and let $\Pi(G)=\{B_0,
B_1, B_2\}$ be the ordered partition of the vertices of $G$ induced
by $f$, where $B_i = \{v\in V\;:\; f(v) = i\}$ and let
$b_i(G)=|B_i|$, with $i\in \{0, 1, 2\}$. Frequently, a Roman dominating function $f$ is represented by its induced partition $\{B_0,
B_1, B_2\}$. It is clear that for any
Roman dominating function $f$ on the graph $G=(V,E)$ of order $n$ we
have that $f(V)=\sum_{u\in V}f(u)=2b_2(G)+b_1(G)$ and
$b_0(G)+b_1(G)+b_2(G)=n$. The following lemma obtained in \cite{roman} will be useful into proving some of the results in this section.

\begin{lemma}{\em \cite{roman}}\label{lema-roman}
For any graph $G$, $\gamma(G)\le \gamma_R(G)\le 2\gamma(G)$.
\end{lemma}

\begin{theorem}
Let $G$ and $H$ be two graphs of order $n_1$ and $n_2\ge 2$,
respectively. Then,
$$\gamma_R(G\odot H)=2n_1.$$
\end{theorem}

\begin{proof}
Let $f$ be a Roman dominating function on $G\odot H$ and let $v_i$
be a vertex of $G$. We have the following cases.

Case 1: $f(v_i)=0$ or $f(v_i)=1$. If there is a vertex $v\in V_i$,
such that $f(v)=0$, then there exists other vertex $x\in V_i$ with
$x\sim v$ and $f(x)=2$. On the contrary, if for every $u\in V_i$ we
have that $f(u)\ne 0$, then $f(u)=1$ or $f(u)=2$ for any vertex
$u\in V_i$. As a consequence, since $n_2\ge 2$, in both cases we
have that $f(V_i\cup \{v_i\})=\sum_{u\in V_i\cup \{v_i\}}f(u)\ge 2$.

Case 2: $f(v_i)=2$. It is clear that $f(V_i\cup \{v_i\})=\sum_{u\in
V_i\cup \{v_i\}}f(u)\ge 2$.

Thus, we obtain that $$\gamma_R(G\odot H)=\sum_{v\in
{V\cup\left(\cup_{i=1}^{n_1}V_i\right)}}f(v)=\sum_{i=1}^{n_1}\sum_{v_\in
V_i\cup\{v_i\}}f(v)\ge 2n_1.$$

On the other side, since $\gamma(G\odot H)=n_1$, by Lemma
\ref{lema-roman} we have that $$\gamma_R(G\odot H)\le 2\gamma(G\odot
H)=2n_1.$$ Therefore, the result follows.
\end{proof}

Next we analyze the corona graph $G\odot K_1$.

\begin{theorem}
Let $G$ be a graph of order $n$. Then there exists a Roman dominating function $\{B_0, B_1, B_2\}$ of minimum weight such that
$$\gamma_R(G\odot K_1)=\gamma_R(G)+n-b_2(G).$$
\end{theorem}

\begin{proof}
Let $f'$ be a Roman dominating function on $G$ of minimum weight.
Let $v_i$ be a vertex of $G$ and we will denote by $u_i$ the pendant
vertex of $v_i$ corresponding to the $i^{th}$ copy of $K_1$ in
$G\odot K_1$. Hence, we define a function $f$ on $G\odot K_1$ in the
following way:
\begin{itemize}
\item For every $v_i\in V$, we have that $f(v_i)=f'(v_i)$.
\item If $f'(v_i)=0$ or $f'(v_i)=1$ for some $i\in \{1,...,n\}$, then $f(u_i)=1$.
\item If $f'(v_j)=2$ for some $j\in \{1,...,n\}$, then $f(u_j)=0$.
\end{itemize}
Thus, it is clear that $f$ is a Roman dominating function for $G\odot K_1$ and the weight of $f$ is given by
\begin{align*}
f(V\cup(\cup_{i=1}^n\{u_i\}))&=\sum_{u\in
V\cup(\cup_{i=1}^n\{u_i\})}f(u)\\
&=f'(V)+f(\cup_{i=1}^n\{u_i\})\\
&=2b_2(G)+b_1(G)+f\left(\cup_{i=1}^n\{u_i\}\right)\\
&=2b_2(G)+2b_1(G)+b_0(G).
\end{align*}
Now, since $\gamma_R(G)=2b_2(G)+b_1(G)$ and $b_0(G)+b_1(G)+b_2(G)=n$
we obtain that
$$\gamma_R(G\odot K_1)\le f(V\cup(\cup_{i=1}^n\{u_i\}))=\gamma_R(G)+n-b_2(G).$$

On the other side, let $h$ be a Roman dominating function on $G\odot
K_1$ of minimum weight and let $h'$ be a function on $G$ such that
for every $v_i\in V$ we have that $h'(v_i)=h(v_i)$. Now, if
$h'(v_l)=0$, for some $l\in \{1,...,n\}$, then as for every vertex
$u_i$, $i\in \{1,... ,n\}$, it is satisfied that $h(u_i)\ne 2$,
there exists $v_j\in N_G(v_l)$, $j\ne l$, such that $h(v_j)=2$. So,
$h'(v_j)=2$ and $h'$ is a Roman dominating function on $G$.

Moreover, we have the following facts:
\begin{itemize}
\item If $h(u_j)=0$ for some $j\in \{1,...,n\}$, then $h(v_i)=h'(v_i)=2$.
\item If $h(u_l)=1$ for some $l\in \{1,...,n\}$, then $h(v_l)=h'(v_l)=0$ or
$h(v_l)=h'(v_l)=1$.
\end{itemize}
Thus, we obtain that
\begin{align*}
\gamma_R(G\odot K_1)&=\sum_{u\in V\cup(\cup_{i=1}^n\{u_i\})}h(u)\\
&=\sum_{u\in V}h(u)+\sum_{u\in \cup_{i=1}^n\{u_i\}}h(u)\\
&=\sum_{u\in V}h'(u)+\sum_{u\in \cup_{i=1}^n\{u_i\}}h(u)\\
&\ge \gamma_R(G)+\sum_{u\in \cup_{i=1}^n\{u_i\}}h(u)\\
&=\gamma_R(G)+b_0(G)+b_1(G)\\
&=\gamma_R(G)+n-b_2(G).
\end{align*}
Therefore, the result follows.
\end{proof}

Notice that the above result gives a formula for the Roman
domination number of $G\odot K_1$, but such a formula depends on
$b_2(G)$, which is unknown in general. Thus, the formula leads to
the conclusion that obtaining the Roman domination number of $G\odot
K_1$ could be very difficult even if it is knew the Roman domination
number of $G$.

\begin{corollary}
Let $G$ be a graph of order $n\ge 2$, different from $\overline{K_n}$. Then,
$$\gamma_R(G)+\frac{n}{2}\le \gamma_R(G\odot K_1)\le \gamma_R(G)+n-1.$$
\end{corollary}

\begin{proof}
The result follows from the above theorem and the fact that $1\le
b_2(G)\le \frac{n}{2}$.
\end{proof}

Notice that if $G$ is the star graph $S_{1,n}$, then the upper bound
of the above theorem is attained.

\section{Location-domination}

The concepts about resolvability and location in graphs were introduced independently by Harary and Melter \cite{harary} and Slater
\cite{leaves-trees}, respectively, to define the same structure in a graph. In the present work we will use the terminology of \cite{leaves-trees}. The domination parameter related to location in graphs can be seen in two different ways as it was presented in \cite{loc-dom-1,loc-dom-2} and \cite{res-dom}, respectively. Given a set $S=\{v_1,v_2,...,v_t\}$ of vertices of a graph $G$, we say that $S$ is a {\em locating set} (or resolving set) if for every two different vertices $u,v$ of $G$ it is satisfied that $(d(u,v_1),d(u,v_2),...,d(u,v_t))\ne (d(v,v_1),d(v,v_2),...,d(v,v_t))$ where $d(x,v_i)$ represents the distance between the vertices $x$ and $v_i$, for every $x\in \{u,v\}$ and $i\in \{1,...,t\}$. The minimum cardinality of any locating set of $G$ is called the {\em location number} (or {\em metric dimension}) of $G$ and it is denoted by $dim(G)$.

Also, a set $D$ of vertices of $G$ is {\em locating dominating} (or resolving dominating) if it is locating (or resolving) and dominating. The minimum cardinality of any locating dominating set of $G$ is called the {\em location domination number} of $G$ and it is denoted by $\gamma_{ld}$. On the other hand the set $D$ is {\em locating-dominating} if for every pair of different vertices $u,v\in \overline{D}$ it is satisfied that $N_D(u)\ne N_D(v)$. The minimum cardinality of any locating-dominating set of $G$ is called the {\em location-domination number} of $G$ and it is denoted by $\gamma_{l-d}(G)$. In this sense the following inequalities chain is satisfied for any connected graph $G$.

$$dim(G)\le \gamma_{ld}(G)\le \gamma_{l-d}(G).$$

For the case of corona graphs, as it was studied in \cite{dim-corona}, we have that every locating set is also a locating dominating set of $G$. Thus, we obtain that $dim(G)=\gamma_{ld}(G)$. So, in the present section we will be centered into studying the locating-dominating number of corona graphs.

\begin{lemma}\label{lema-loc-dom-1}
Let $G$ and $H$ be connected graphs. Let $S$ be a locating-dominating set of minimum cardinality in $G\odot H$ and let $v_i$ be a vertex of $G$. Then,
\begin{itemize}
\item[{\rm (i)}] If $v_i\notin S$, then $S\cap V_i$ is a locating-dominating set in the copy $H_i$ of $H$ in $G\odot H$.
\item[{\rm (ii)}] If $v_i\in S$, then $S\cap (V_i\cup \{v_i\})$ is a locating-dominating set in the subgraph $K_1\odot H_i$.
\end{itemize}
\end{lemma}

\begin{proof}
Let $S_i=S\cap V_i$. If $v_i\notin S$, then as every vertex $x\in V_i$ is adjacent to only one vertex not in $V_i$ (the vertex $v_i$) it is satisfied that $N_S(x)=N_{S_i}(x)$. So, for every two different vertices $u,v\in \overline{S_i}$ in $H_i$ we have that $N_{S_i}(u)\ne N_{S_i}(v)$ and, as a consequence, (i) follows.

Now, let us suppose that $v_i\in S$ and let $S'_i=S\cap (V_i\cup\{v_i\})$. Let $v_i$ be the vertex of $K_1$. Hence, for every vertex $x\in \overline{S'_i}$ in $K_1\odot H_i$ it is satisfied that $N_{S'_i}(x)=\{v_i\}\cup (S\cap V_i)=N_S(x)$. Thus, for every two different vertices $u,v\in \overline{S'_i}$ we have
$$N_{S'}(u)=\{v_i\}\cup (S\cap V_i)=N_S(u)\ne N_S(v)= \{v_i\}\cup (S\cap V_i)=N_{S'}(v).$$
Thus, (ii) follows.
\end{proof}

\begin{lemma}\label{lema-loc-dom-2}
For any graph $H$,
\begin{itemize}
\item[{\rm (i)}] If there exist a locating-dominating set $A$ of minimum cardinality in $H$ such that for every vertex $v\in \overline{A}$ it is satisfied that $N_{A}(v)\subsetneq A$, then  $\gamma_{l-d}(K_1\odot H)=\gamma_{l-d}(H)$.
\item[{\rm (ii)}] If for any locating-dominating set $B$ of minimum cardinality in $H$ there exists a vertex $u\in \overline{B}$ such that $N_{B}(u)= B$, then $\gamma_{l-d}(K_1\odot H)=\gamma_{l-d}(H)+1$.
\end{itemize}
\end{lemma}

\begin{proof}
Let $S$ be a locating-dominating set of minimum cardinality in $K_1\odot H$ and let $v$ be the vertex of $K_1$. Now, if there exist a locating-dominating set $A$ of minimum cardinality in $H$ such that for every vertex $v\in \overline{A}$, $N_{A}(v)\subsetneq A$, then since $v\sim u$ for every vertex $u$ of $H$, it is satisfied that $v\notin S$. So, for any two different vertices $x,y\in \overline{S}-\{v\}$ we have that $N_S(x)\ne N_S(y)$. Thus, $S$ is a locating-dominating set in $H$ and $\gamma_{l-d}(K_1\odot H)\ge \gamma_{l-d}(H)$. Now, if there exist a locating-dominating set $A$ of minimum cardinality in $H$ such that for every vertex $v\in \overline{A}$ it is satisfied that $N_{A}(v)\subsetneq A$, then it is clear that $A$ is also a locating-dominating set in $K_1\odot H$. So, $\gamma_{l-d}(K_1\odot H)\le \gamma_{l-d}(H)$ and (i) follows.

On the other side, let us suppose that for every locating-dominating set $B$ of minimum cardinality in $H$ there exists a vertex $u\in \overline{B}$ such that $N_{B}(u)= B$. Since for the vertex $v$ of $K_1$ it is satisfied that $N_B(v)=B$ we obtain that any locating-dominating set of minimum cardinality in $K_1\odot H$ must contain the set $B$ and either the vertex $v$ or the other vertex $u$ of $H$ such that $N_B(u)=B$. So, $\gamma_{l-d}(K_1\odot H)\ge \gamma_{l-d}(H)+1$. On the other hand, if $B'$ is a locating-dominating set of minimum cardinality in $H$ such that there exists a vertex $u'\in \overline{B'}$ with $N_{B'}(u')= B'$, then it is easy to check that $B'\cup \{u'\}$ is a locating-dominating set in $K_1\odot H$. Therefore, $\gamma_{l-d}(K_1\odot H)\le \gamma_{l-d}(H)+1$ and (ii) follows.
\end{proof}

\begin{theorem}\label{valor-loc-dom}
For any connected graph $G$ of order $n$ and any connected graph $H$,
\begin{itemize}
\item[{\rm (i)}] If there exist a locating-dominating set $A$ of minimum cardinality in $H$ such that for every vertex $v\in \overline{A}$ it is satisfied that $N_{A}(v)\subsetneq A$, then  $\gamma_{l-d}(G\odot H)=n\gamma_{l-d}(H).$
\item[{\rm (ii)}] If for any locating-dominating set $B$ of minimum cardinality in $H$ there exists a vertex $u\in \overline{B}$ such that $N_{B}(u)= B$, then $\gamma_{l-d}(G\odot H)=n\gamma_{l-d}(H)+\gamma(G).$
\end{itemize}
\end{theorem}

\begin{proof}
Let $S_i$, $i\in \{1,...,n\}$, be a locating-dominating set of minimum cardinality in the copy $H_i=(V_i,E_i)$ of $H$ in $G\odot H$ such that for every vertex $v\in \overline{S_i}$ it is satisfied that $N_{S_i}(v)\subsetneq S_i$. Now, let $S=\bigcup_{i=1}^nS_i$ and let $x,y\in \overline{S}$ be two different vertices of $G\odot H$. If $x$ is a vertex of $G$ and $y$ is a vertex of a copy $H_i$ of $H$ in $G\odot H$ with $x\sim y$, then we have that
$$N_{S}(y)=N_{S_i}(y)\subsetneq S_i=N_S(x).$$
On the contrary, in any other cases for $x,y$ we have that $N_{S}(y)\ne N_S(x)$. Thus, $S$ is a locating-dominating set in $G\odot H$ and $\gamma_{l-d}(G\odot H)\le n\gamma_{l-d}(H)$.

On the other hand, let $S$ be a locating-dominating set in $G\odot H$. So, by Lemma \ref{lema-loc-dom-1}, for every $i\in \{1,...,n\}$ we have either $|S\cap V_i|\ge \gamma_{l-d}(H)$ or $|S\cap (V_i\cap\{v_i\})|\ge \gamma_{l-d}(K_1\odot H)$. Also, by Lemma \ref{lema-loc-dom-2} we have that $\gamma_{l-d}(K_1\odot H)\ge \gamma_{l-d}(H)$. Therefore, $$\gamma_{l-d}(G\odot H)=|S|\ge \sum_{i=1}^n\gamma_{l-d}(H)=n\gamma_{l-d}(H).$$

Therefore (i) follows. In order to prove (ii), let $S'_i$, $i\in \{1,...,n\}$, be a locating-dominating set in the copy $H_i$ of $H$ in $G\odot H$. Hence, there exists a vertex $u\in \overline{S'_i}$ such that $N_{S'_i}(u)= S'_i$. Let $D$ be a dominating set of minimum cardinality in $G$
and let $S'=D\cup (\bigcup_{i=1}^n S'_i)$. Now, let $x,y\in \overline{S'}$ such that $x$ is a vertex of $G$, $y$ is a vertex of a copy $H_i$ of $H$ and $x\sim y$.  So, if $y=u$, then $N_{S'}(y)=N_{S'_i}(y)$.
Since $D$ is a dominating set in $G$, there exists at least a vertex $v\in D\subset S$ such that $x\sim v$. Thus, $S_i\cup \{v\}\subseteq N_{S'}(x)$ and as a consequence, $N_{S'}(x)\ne N_{S'}(y)$. On the contrary, in any other case for $x,y\in \overline{S'}$ we have that $N_{S'}(x)\ne N_{'}S(y)$. Therefore, $S'$ is a locating-dominating set in $G\odot H$ and $\gamma_{l-d}(G\odot H)\le n\gamma_{l-d}(H)+\gamma(G).$

On the other side, let $S'$ be a locating-dominating set of minimum cardinality in $G\odot H$. For any locating-dominating set $L$ in $H$ there exists a vertex $u\in \overline{L}$ such that $N_L(u)=L$ and also for every $v_i\in V$ in $G\odot H$, $N_L(v_i)=L$. So, for every $v_i\in V$ there is $u_i\in V_i$ such that $N_L(u_i)=N_L(v_i)=L$. Since $S'$ is a locating-dominating set in $G\odot H$ we have either,

\begin{itemize}
\item $u_i\in S'$ and $v_i\notin S'$. In such a case by taking the set $S''=(S'-\{u_i\})+\{v_i\}$ we have that $S''$ is also a locating-dominating set of minimum cardinality in $G\odot H$,

\item or $v_i\in S'$ and $u_i\notin S'$.
\end{itemize}

Thus, let $A$ be the set of vertices of $G$ such that for every vertex $v_j\in A$ we have either $v_j\in S$ or there exists $u_j\in V_j$ such that $u_j\in S$ and $N_L(u_j)=L$ for any locating-dominating set $L$ of minimum cardinality in $H_j$. Thus, $A$ is a dominating set in $G$ and we obtain that
\begin{align*}
|S|&=\sum_{i=1}^n(S\cap (V_i\cup \{v_i\}))\\
&=\sum_{i=1}^{|A|}(S\cap (V_i\cup \{v_i\}))+\sum_{j=1}^{n-|A|}(S\cap V_i)\\
&\ge |A|(\gamma_{l-d}(H)+1)+(n-|A|)\gamma_{l-d}(H)\;\;(\mbox{By Lemma \ref{lema-loc-dom-2}})\\
&=|A|+n\gamma_{l-d}(H)\\
&=\gamma(G)+n\gamma_{l-d}(H).
\end{align*}
Therefore, the proof of (ii) is complete.
\end{proof}

\section{Other kinds of domination related parameters}

A set of vertices $D$ of a graph $G$ is a ({\em
connected}\footnote{$D$ is connected in $G$ if for any two different
vertices $u,v\in D$ there exists a path $P$ of length $d(u,v)$
between $u$ and $v$ such that every vertex of $P$ belongs to $D$.},
{\em convex}\footnote{$D$ is convex in $G$ if for any two different
vertices $u,v\in D$ all the vertices of all paths of length $d(u,v)$
between $u$ and $v$ belong to $D$.} or {\em independent}) {\em
dominating set} in $G$ if $D$ is a dominating set and a (connected,
convex or independent) set in $G$. The minimum cardinality of any
(connected, convex or independent) dominating set in $G$ is called
the ({\em connected, convex} or {\em independent}) {\em domination
number} of $G$ and it is denoted by ($\gamma_c(G), \gamma_{con}(G)$
or $i(G)$). A set $D$ is a {\em distance}-$k$ {\em dominating set}
in $G$ if for every vertex $v\in \overline{D}$ it follows that
$d(u,v)\le k$ for some $v\in D$, where $d(u,v)$ represents the
distance between the vertices $u$ and $v$. The minimum cardinality
of any distance-$k$ dominating set of a graph is called the {\em
distance-$k$ dominating} number of $G$ and it is denoted by
$\gamma_{\le k}(G)$.

There are some domination parameters whose value is very easy to
observe for  the case of corona graph. For instance, it is clear
that $\gamma(G\odot H)=\gamma_c(G\odot H)=\gamma_{con}(G\odot
H)=n_1$ and $\beta_0(G\odot H)=n_1\beta_0(H)$. At next we obtain the exact value of some domination related parameters of corona graphs.

\begin{theorem}
For any connected graph $G$ of order $n$ and for any graph $H$, if $k\ge 2$, then
$$\gamma_k(G\odot H)=n\min\{\gamma_k(H),\gamma_{k-1}(H)+1\}.$$
\end{theorem}

\begin{proof}
Let $S$ be a $k$-dominating set of minimum cardinality in $G\odot
H$. If $v_i\in V\cap S$, then for every  $v\in \overline{V_i\cap S}$
in $H_i$ we have that $k\le\delta_S(v)=\delta_{V_i\cap S}(v)+1$.
Thus, $V_i\cap S$ is a $(k-1)$-dominating set in $H_i$. Also, if
$v_j\notin V\cap S$, then for every $v\in \overline{V_j\cap S}$ in
$H_j$ we have that $k\le\delta_S(v)=\delta_{V_j\cap S}(v)$ and we
obtain that $V_j\cap S$ is a $k$-dominating set in $H_j$. Let
$A=S\cap V$. Hence, we have that

\begin{equation}\label{equation-1}
|S|\ge |A|+|A|\gamma_{k-1}(H)+(n-|A|)\gamma_{k}(H).
\end{equation}

Now, if $\gamma_{k-1}(H)=\gamma_k(H)$ or
$\gamma_{k-1}(H)=\gamma_k(H)-1$, then by (\ref{equation-1}) we
obtain that $|S|\ge n\gamma_{k}(H)$. On the contrary, if
$\gamma_{k-1}(H)\le \gamma_k(H)-2$ then by (\ref{equation-1}) we
have

\begin{align*}
|S|&\ge |A|+|A|\gamma_{k-1}(H)+(n-|A|)\gamma_{k}(H)\\
&\ge |A|+|A|\gamma_{k-1}(H)+(n-|A|)(\gamma_{k-1}(H)+2)\\
&=2n-|A|+n\gamma_{k-1}(H)\\
&\ge 2n-n+n\gamma_{k-1}(H)\\
&=n(\gamma_{k-1}(H)+1).
\end{align*}
Therefore, we obtain that $\gamma_k(G\odot H)\ge
n\min\{\gamma_k(H),\gamma_{k-1}(H)+1\}$.

On the other hand,  let $A_i$
be a $k$-dominating  set in $H_i$, $i\in \{1,...,n\}$ and let
$A=\bigcup_{i=1}^{n}A_i$. So, for every vertex $v_i\in V$ we have
that $\delta_{A}(v_i)=\delta_{A_i}(v_i)=|A_i|\ge k$. Also, for
every vertex $u\in \overline{A_i}$ in $H_i$ we have that
$\delta_A(u)=\delta_{A_i}(v)\ge k$. Thus, $A$ is a $k$-dominating
set in $G\odot H$ and, as a consequence, $\gamma_k(G\odot H)\le
n\gamma_k(H)$.

Also, let $B_i$ be a
$(k-1)$-dominating set in $H_i$,  $i\in\{1,...,n\}$ and let
$B=(\bigcup_{i=1}^{n}B_i)\cup V$. So, for every vertex $u\in
\overline{B_i}$ in $H_i$ we have that
$\delta_B(u)=\delta_{B_i}(u)+1\ge k$. Thus, $B$ is a $k$-dominating
set in $G\odot H$ and, as a consequence, $\gamma_k(G\odot H)\le
n+n\gamma_{k-1}(H)=n(\gamma_{k-1}(H)+1)$.

Therefore, $\gamma_k(G\odot H)\le n\min\{\gamma_k(H),(\gamma_{k-1}(H)+1)\}$ and the result follows.
\end{proof}

\begin{theorem}
For any connected graph $G$ and any graph $H$, if $k\ge 2$, then
$$\gamma_{\le k}(G\odot H)=\gamma_{\le k-1}(G).$$
\end{theorem}

\begin{proof}
Let $S$ be a distance-$(k-1)$ dominating set in $G$ with order $n$. Hence, for
every vertex $v_i\in V$, with $i\in \{1,...,n\}$, we  have that $d_{G\odot
H}(v_i,S)=d_G(v_i,S)\le k-1$. Also, for every vertex $u\in V_i$, $i\in \{1,...,n\}$, in
$G\odot H$ we have that $d_{G\odot H}(u,S)=d_{G}(v_i,S)+1\le k$. So, $S$ is a
distance-$k$ dominating set in $G\odot H$ and, as a consequence,
$\gamma_{\le k}(G\odot H)\le \gamma_{\le k-1}(G)$.

On the other hand, let $B$ be a distance-$k$ dominating set in
$G\odot H$ of minimum cardinality. Now, let  $A=\{v_{i_1},v_{i_2},...v_{i_r}\}$ be the set
of vertices of $G$ such that $(V_{i_j}\cup \{v_{i_j}\})\cap B\ne
\emptyset$, for every $j\in \{1,...,r\}$. Since, $\gamma_{\le
k}(G\odot H)\le \gamma_{\le k-1}(G)<n$, we have that $r\le n-1$
and for every vertex $v_l\in\overline{A}$ in $G$ we have,
$$d_G(v_l,A)\le d_{G\odot H}(v_l,B)\le k.$$
Now, if $d_G(v_l,A)=k$, then $d_G(v_l,A)=d_{G\odot H}(v_l,B)=k$ and for any vertex $u\in V_l$ we have that
$$d_{G\odot H}(u,B)=d_{G\odot H}(u,v_l)+d_{G\odot H}(v_l,B)=d_{G}(v_l,A)+1=k+1,$$
which is a contradiction because $B$ is a distance-$k$ dominating
set in $G\odot H$. Thus,  for every vertex $v_l\in \overline{A}$ in
$G$, we have that $d_G(v_l,A)\le k-1$ and, as a consequence, $A$ is
a distance-$(k-1)$ dominating set in $G$.

Therefore, $\gamma_{\le k}(G\odot H)=|B|\ge |A|\ge \gamma_{\le k-1}(G)$ and the result follows.
\end{proof}

\begin{theorem}
For any connected graph $G$ of order $n$ and for any graph $H$,
$$i(G\odot H)=ni(H)-\beta_0(G)(i(H)-1).$$
\end{theorem}

\begin{proof}
Let $S$ be an independent dominating set of minimum cardinality in $G\odot H$. If $v_i\in
V\cap S$, then for every $v\in V_i$ we have that $v\notin
S$. Also, if $v_i\notin V\cap S$, then there exists
$S_i\subset V_i$, such that $S_i\subset S$ and $|S_i|\ge i(H)$.
Thus, we obtain that there exist the sets $A\subset V$ and
$S_i\subset V_i$, $i\in \{1,...,t\}$, such that $n=t+|A|$, $|A|\le
\beta_0(G)$ and $S=(\bigcup_{i=1}^tS_i)\cup A$. Thus, $t\ge
n-\beta_0(G)$ and we have
\begin{align*}
|S|&=|A|+\sum_{i=1}^t|S_i|\\
&=n-t+\sum_{i=1}^t|S_i|\\
&\ge n-t+ti(H)\\
&=n+t(i(H)-1)\\
&\ge n+(n-\beta_0(G))(i(H)-1)\\
&=ni(H)-\beta_0(G)(i(H)-1).
\end{align*}
Therefore, $i(G\odot H)=|S|\ge ni(H)-\beta_0(G)(i(H)-1).$

On the other hand, let $A$ be an independent set of maximum
cardinality in $G$. Now, for every $v_i\in \overline{A}$, let
$S_i\subset V_i$ be an independent dominating set in $H_i$. Let
$S=A\cup (\bigcup_{v_i\in \overline{A}}S_i)$. It is easy to see that
$S$ is independent and dominating. So, $i(G\odot H)\le
ni(H)-\beta_0(G)(i(H)-1)$ and the result follows.
\end{proof}

A {\em domatic partition} of a graph $G$ is a vertex partition of $G$ in
which every set is a  dominating set \cite{domatic-1,domatic-2,domatic-3}. The maximum number of sets in
any domatic partition of $G$ is called the {\em domatic number} of $G$ and
it is denoted by  $d(G)$. If $G$ has a domatic partition, then $G$
is a {\em domatic graph} (or $G$ is domatic). Similarly, if there
exists a vertex partition of $G$ into independent dominating sets,
then such a partition is called {\em idomatic} \cite{domatic-3}. The maximum number
of sets in any partition of a graph $G$ into independent dominating
sets is called the {\em idomatic number} of $G$ and it is denoted by
$d_i(G)$.  If $G$ has an idomatic partition, then $G$ is an {\em
idomatic graph} (or $G$ is idomatic). Next we study the domatic and
idomatic numbers of corona graphs.

\begin{remark}
For any connected graph $G$ and for any graph $H$, $$d(G\odot
H)=d(H)+1.$$
\end{remark}

\begin{proof}
Let $\Pi_i=\{S_{i1},S_{i2},...,S_{id(H)}\}$ be a domatic partition
for $H_i$. Now, let $S_i=\cup_{j=1}^{n_1}S_{ji}$, where $n_1$ is the
order of $G$ and $i\in \{1,...,d(G)\}$. Since every $S_{ij}$, $j\in
\{1,...,d(H)\}$ is a dominating set in $H_i$, $i\in \{1,...,n_1\}$
we obtain that $S_i=\cup_{j=1}^{n_1}S_{ji}$, $i\in \{1,...,n_1\}$ is
a dominating set in $G\odot H$. Also, as $V$ is a dominating set in
$G\odot H$ we have that $d(G\odot H)\ge d(H)+1$.

On the other hand, let $\Pi=\{A_1,A_2,...,A_t\}$ be a domatic
partition of maximum cardinality for $G\odot H$. If $v_l\in V$
belongs to $A_j\in \Pi$, then $A_j\cap V_l=\emptyset$ and for every
$A_i\in \Pi$, $i\ne j$, $A_i\cap V_l\ne \emptyset$ and also $A_i\cap
V_l$ is an independent dominating set in $H_l$ for every $i\in
\{1,...,r\}-\{j\}$. Thus,  $\Pi'=\{A_1\cap V_l,A_2\cap
V_l,...,A_{j-1}\cap V_l,A_{j+1}\cap V_l,...,A_r\cap V_l\}$ is a
domatic partition for $H_l$. Therefore $d(G\odot H)=t\le d(H)+1$ and
the result follows.
\end{proof}

\begin{theorem}
Let $G$ be a connected graph and let $H$ be an idomatic graph. Then
$d_i(G\odot H)=d_i(H)+1$ if and only if $G$ has a partition into
$d_i(H)+1$ independent sets.
\end{theorem}

\begin{proof}
Since $v_i\sim u$ for every  $u\in V_i$, $i\in \{1,...,n\}$, where $n$ is the order of $G$, we have that for every independent
dominating set $S$ it is satisfied that $v_i\in S$ if and only if
$V_i\cap S=\emptyset$.

($\Leftarrow$) Let us suppose that $G$ has a partition into
$t=d_i(H)+1$ independent sets and let $\{A_1,A_2,..., A_t\}$ be the
partition of $G$ into $t$ independent sets. Now, for every
$v_{i_j}\in A_i$, $i\in \{1,...,t\}$, let
$\{B_{i_j1},B_{i_j2},...,B_{i_j,i_j-1},B_{i_j,i_j+1},...,B_{i_jt}\}$
be an idomatic partition of $H_{i_j}$.

Let us form a partition $\Pi=\{S_1,S_2,...,S_t\}$ of $G\odot H$ such
that
$$S_i=A_i\cup\left(\bigcup_{v_l\notin A_i}B_{li}\right).$$
Thus, every $S_i\in \Pi$ is an independent dominating set in $G\odot
H$ and, a consequence, $\Pi$ is an idomatic partition in $G\odot H$
and $d_i(G\odot H)\ge t=d_i(H)+1$.

Now, let $\Pi=\{S_1,S_2,...,S_r\}$ be an idomatic partition of maximum cardinality in
$G\odot H$. Let $v_i\in V$ be a vertex of $G$. Hence, there exists
$S_l\in \Pi$ such that $v_i\in S_l$ and $S_l\cap V_i=\emptyset$.
Moreover, for every $S_j\in \Pi$, with $j\ne l$, we have that
$S_j\cap V_i\ne \emptyset$. Let $\Pi_i=\{S'_1,S'_2,...,
S'_{l-1},S'_{l+1},...,S_r\}$ the partition of $H_i$ obtained from
$\Pi$ in such a way that $S'_{j}=S_j\cap V_i$, for every $j\in
\{1,...,r\}-\{l\}$. Since every vertex of $H_i$ is not adjacent to
any vertex outside of $V_i\cup \{v_i\}$ we have that $\Pi$ is an
idomatic partition in $H_i$. Thus, we have
$$d_i(H)\ge r-1=d_{i}(G\odot H)-1.$$
Therefore, we obtain that $d_i(G\odot H)=d_i(H)+1$.

($\Rightarrow$) Let $\Pi=\{S_1,S_2,...,S_r\}$ be an idomatic partition in $G\odot H$, with $r=d_i(G\odot
H)=d_i(H)+1$. If there exists
$S_j\in \Pi$ such that $S_j\cap V=\emptyset$, then we have that for
every $S_i\in \Pi$, $S_i\cap V_j\ne \emptyset$. Since,
$\Pi_j=\{S_1\cap V_j,S_2\cap V_j,..., S_r\cap V_j\}$ is an idomatic
partition in $H_j$ we obtain a contradiction. So, for every $S_i\in
\Pi$, $S_i\cap V\ne \emptyset$. As every $S_i\in \Pi$ is an
independent dominating set we have that $\{S_1\cap V,S_2\cap V,...,
S_r\cap V\}$ is a partition of $V$ into $r=d_i(H)+1$ independent
sets.
\end{proof}

\section*{Appendix}

In this extra section we present some results which are useful into proving some of the above theorems or propositions. The girth $g(G)$ of the graph $G$ is the length of a shortest cycle contained in $G$.

\begin{lemma}\label{lema-vecinos-dist-t}
Let $G$ be a graph of minimum degree $\delta$ and let $t\ge 1$ be an integer. If $g(G)\ge 2t+1$
then, for any vertex $v\in V$
$$|M_t[v]|\ge 1+\delta(v)\displaystyle\sum_{i=0}^{t-1}(\delta-1)^i.$$
\end{lemma}

\begin{proof}
Let $v\in V$ be a vertex of $G$. Hence, we have that
$|M_1[v]|=1+\delta(v)$ and
\begin{align*}
|M_2[v]|&=|M_1[v]|+\sum_{u\in M_1[v]-\{v\}}(\delta(u)-1)\\
&\ge 1+\delta(v)+\sum_{u\in M_1[v]-\{v\}}(\delta-1)\\
&= 1+\delta(v)+\delta(v)(\delta-1)\\
&=1+\delta(v)\displaystyle\sum_{i=0}^{1}(\delta-1)^i
\end{align*}
Now, let us proceed by induction on $t$. Let us suppose that
$|M_{t-1}[v]|\ge
1+\delta(v)\displaystyle\sum_{i=0}^{t-2}(\delta-1)^i$, hence
\begin{align*}
|M_t[v]|&=|M_{t-1}[v]|+\sum_{u\in M_{t-1}[v]-M_{t-2}[v]}(\delta(u)-1)\\
&\ge 1+\delta(v)\displaystyle\sum_{i=0}^{t-2}(\delta-1)^i+\sum_{u\in M_{t-1}[v]-M_{t-2}[v]}(\delta(u)-1)\\
&\ge 1+\delta(v)\displaystyle\sum_{i=0}^{t-2}(\delta-1)^i+\sum_{u\in M_{t-1}[v]-M_{t-2}[v]}(\delta-1)\\
&\ge 1+\delta(v)\displaystyle\sum_{i=0}^{t-2}(\delta-1)^i+\delta(v)(\delta-1)^{t-1}\\
&= 1+\delta(v)\displaystyle\sum_{i=0}^{t-1}(\delta-1)^i
\end{align*}
\end{proof}

\begin{lemma}\label{lema-edges-dist-t}
Let $G$ be a graph of minimum degree $\delta$ and let $t\ge 1$ be an integer. If $g(G)\ge 2t+2$
then, for any edge $uv\in E$
$$|M_t[u]\cup M_t[v]|\ge\left\{\begin{array}{cc}
                          2+2\delta\displaystyle\sum_{i=1}^{t/2}(\delta-1)^{2i-1}, & \mbox{ if $\;t$ is even,} \\
                          2\delta\displaystyle\sum_{i=0}^{(t-1)/2}(\delta-1)^{2i}, & \mbox{ if $\;t$ is odd.}
                        \end{array}\right.
$$
\end{lemma}

\begin{proof}
Let $uv\in E$ be an edge of $G$. By Lemma \ref{lema-vecinos-dist-t}
we have that
$$|M_t[u]|\ge 1+\delta(u)\displaystyle\sum_{i=0}^{t-1}(\delta-1)^i\ge 1+\delta\displaystyle\sum_{i=0}^{t-1}(\delta-1)^i$$
and
$$|M_t[v]|\ge 1+\delta(v)\displaystyle\sum_{i=0}^{t-1}(\delta-1)^i\ge 1+\delta\displaystyle\sum_{i=0}^{t-1}(\delta-1)^i.$$
Since $g(G)\ge 2t+2$ and $u\sim v$ we also have that
\begin{align*}
|M_t[u]\cup M_t[v]|&=|M_t[u]|+|M_t[v]|-|M_t[u]\cap M_t[v]|\\
&=|M_t[u]|+|M_t[v]|-\left(|M_{t-1}[u]|+|M_{t-1}[v]|-|M_{t-1}[u]\cap M_{t-1}[v]|\right)\\
&=|M_t[u]|+|M_t[v]|-|M_{t-1}[u]|-|M_{t-1}[v]|+\left(|M_{t-2}[u]|+|M_{t-2}[v]|-|M_{t-3}[u]\cap M_{t-3}[v]|\right)\\
&=\;\;\;\; ..........................................................\\
&=|M_t[u]|+|M_t[v]|-|M_{t-1}[u]|-|M_{t-1}[v]|+|M_{t-2}[u]|+|M_{t-2}[v]|-\; ...\\
&\hspace*{0.5cm}+(-1)^{t-2}|M_{2}[u]|+(-1)^{t-2}|M_{2}[v]|+(-1)^{t-1}|M_{1}[u]|+(-1)^{t-1}|M_{1}[v]|+(-1)^{t}2\\
&\ge
2+2\delta\displaystyle\sum_{i=0}^{t-1}(\delta-1)^i-2-2\delta\displaystyle\sum_{i=0}^{t-2}(\delta-1)^i+
2+2\delta\displaystyle\sum_{i=0}^{t-3}(\delta-1)^i-2-2\delta\displaystyle\sum_{i=0}^{t-4}(\delta-1)^i+\;...\\
&\hspace*{0.5cm}+(-1)^{t-2}\left(2+2\delta\displaystyle\sum_{i=0}^{1}(\delta-1)^i\right)+(-1)^{t-1}\left(2+2\delta\displaystyle\sum_{i=0}^{0}(\delta-1)^i\right)+
(-1)^{t}2
\end{align*}
Now, if $t$ is even, then we obtain
\begin{align*}
|M_t[u]\cup M_t[v]|&\ge 2\delta(\delta-1)^{t-1}+2\delta(\delta-1)^{t-3}+\;...\;+2\delta(\delta-1)^3+2\delta(\delta-1)+2\\
&=2+2\delta\displaystyle\sum_{i=1}^{t/2}(\delta-1)^{2i-1}.
\end{align*}
On the contrary, if $t$ is odd, then we have
\begin{align*}
|M_t[u]\cup M_t[v]|&\ge 2\delta(\delta-1)^{t-1}+2\delta(\delta-1)^{t-3}+\;...\;+2\delta(\delta-1)^2+2\delta\\
&=2\delta\displaystyle\sum_{i=0}^{(t-1)/2}(\delta-1)^{2i}.
\end{align*}
\end{proof}

\begin{theorem}\label{cota-k-par}
Let $G$ be a graph of minimum degree $\delta$ and maximum degree
$\Delta$. Let $k\ge 2$ be an integer. If $g(G)\ge k+1$ and $k$ is even, then
$$\chi_{\le k}(G)\ge 1+\Delta\displaystyle\sum_{i=0}^{\frac{k}{2}-1}(\delta-1)^i.$$
\end{theorem}

\begin{proof}
Let us suppose $k$ is even and let $v$ be a vertex of maximum degree
in $G$. Let $A\subset V$ be the set of vertices of $G$ such that for
every $u\in A$ we have $d(u,v)\le \frac{k}{2}$. Now, since $g(G)\ge
k+1$, by Lemma \ref{lema-vecinos-dist-t} we have that $|A|\ge
1+\Delta\sum_{i=0}^{k/2-1}(\delta-1)^i$. Also, for every two
vertices $x,y\in A$ we have that $d(x,y)\le k$. So, we obtain that
$c(x)\ne c(y)$ and, as a consequence, $\chi_{\le k}(G)\ge |A|$.
Thus, the result follows.
\end{proof}

\begin{theorem}
Let $G$ be a graph of minimum degree $\delta$ and maximum degree
$\Delta$. Let $k\ge 2$ be an integer. If $g(G)\ge k+1$ and $k$ is odd, then
$$\chi_{\le k}(G)\ge\left\{\begin{array}{cc}
                          2+2\delta\displaystyle\sum_{i=1}^{(k-1)/4}(\delta-1)^{2i-1}, & \mbox{ if $\;\frac{k-1}{2}$ is even,} \\
                          2\delta\displaystyle\sum_{i=0}^{(k-3)/4}(\delta-1)^{2i}, & \mbox{ if $\;\frac{k-1}{2}$ is odd.}
                        \end{array}\right.$$
\end{theorem}

\begin{proof}
Let us suppose $k$ is odd and let $uv\in E$ be an edge of $G$. Let
$B\subset V$ be the set of vertices of $G$ such that for every $x\in
B$ we have either $d(x,v)\le \frac{k-1}{2}$ or $d(x,u)\le
\frac{k-1}{2}$. Let $r=\frac{k-1}{2}$. Since, $g(G)\ge k+1$, by Lemma
\ref{lema-edges-dist-t} we obtain that
$$|A|\ge\left\{\begin{array}{cc}
                          2+2\delta\displaystyle\sum_{i=1}^{r/2}(\delta-1)^{2i-1}, & \mbox{ if $\;r$ is even,} \\
                          2\delta\displaystyle\sum_{i=0}^{(r-1)/2}(\delta-1)^{2i}, & \mbox{ if $\;r$ is odd.}
                        \end{array}\right.$$
Now, for every two vertices $a,b\in B$ we have $d(a,b)\le k$. Thus,
$c(a)\ne c(b)$ and, as a consequence, $\chi_{\le k}(G)\ge |B|$.
Therefore, the result follows.
\end{proof}


\begin{thebibliography}{99}

\bibitem{corona-spectrum} S. Barik, S. Pati, B. K. Sarma, The spectrum of the corona of two graphs,
{\em SIAM Journal on Discrete Mathematics} {\bf 21} (1) (2007) 47--56.

\bibitem{bollobas} B. Bollob\'as, D. B. West, A note on generalized chromatic number and generalized girth,
 {\em Discrete Mathematics} {\bf 213} (2000) 29--34.

\bibitem{res-dom} R. C. Brigham, G. Chartrand, R. D. Dutton, P. Zhang, Resolving domination in graphs, {\em Mathematica Bohemica} {\bf 128} (1) (2003) 25--36.

\bibitem{color-aplic-1} G. J. Chaitin, Register allocation and spilling via graph colouring, {\em Proceeding 1982 SIGPLAN Symposium on
Compiler Construction} (1982) 98--105. doi:10.1145/800230.806984

\bibitem{chartrand} G. Chartrand, A scheduling problem: An introduction to chromatic numbers,
In $\S$ 9.2 {\em Introductory Graph Theory}. New York: Dover, pp. 202--209 (1985).

\bibitem{corona-banwidth} P. Z. Chinn, Y. Lin, J. Yuan, The bandwidth of the corona of two graphs,
{\em Congressus Numerantium} {\bf 91} (1992) 141-152.

\bibitem{domatic-1} E. J. Cockayne, S. T. Hedetniemi, Towards a theory of domination in graphs, {\em Networks} {\bf 7} (1977) 247--261.

\bibitem{roman} E. J. Cockayne, P. A. Dreyer, S M. Hedetniemi, S. T. Hedetniemi, Roman domination in graphs, {\em Discrete Mathematics} {\bf
278} (1-3) (2004) 11--22.

\bibitem{dist-k-color-1} G. Fertin, E. Godard, A. Raspaud, Acyclic and $k$-distance coloring of the grid,
{\em Information Processing Letters} {\bf 87} (1) (2003) 51--58.

\bibitem{corona-first} R. Frucht, F. Harary,  On the corona of two graphs, {\em Aequationes Mathematicae} {\bf 4} (1970) 322--325.

\bibitem{dist-k-color-2} J. P. Georges, D. M. Mauro, M. I. Stein, Labeling products of complete graphs
with a condition at distance two, {\em SIAM Journal on Discrete Mathematics} {\bf 14} (1) (2000) 28--35.

\bibitem{dist-k-color-3} W. K. Hale, Frequency assignment: theory and applications, Proceedings IEEE 68 (1980), 1497--1514.

\bibitem{harary} F. Harary, R. A. Melter, On the metric dimension of a graph, {\it Ars Combinatoria} {\bf 2} (1976)
191--195.

\bibitem{bookdom1}T. W. Haynes, S. T. Hedetniemi, P. J. Slater, \emph{Fundamentals of Domination
in Graphs}, Marcel Dekker, Inc. New York, 1998.

\bibitem{bookdom2}T. W. Haynes, S. T. Hedetniemi, P. J. Slater, \emph{Domination in graphs: Advanced topics}, Marcel Dekker, Inc. New York, 1998.

\bibitem{hedet-conj-first} S. T. Hedetniemi, Homomorphisms of graphs and automata. Technical Report 03105-44-T. University of Michigan. (1966)

\bibitem{corona-profile} Y. L. Lai, G. J. Chang, On the profile of the corona of two graphs,
{\em Information Processing Letters} {\bf 89} (6) (2004) 287--292.

\bibitem{lawler}  E. L. Lawler, A note on the complexity of the chromatic number problem,
{\em Information Processing Letters} {\bf 5} (3) (1976) 66--67.

\bibitem{color-aplic-2} D. Marx, Graph colouring problems and their applications in scheduling, {\em Periodica Polytechnica-Electrical Engineering}
{\bf48} (2004) 11--16.

\bibitem{hedet-conjec-survey} N. Sauer, Hedetniemi's conjecture: a survey, {\em Discrete Mathematics} {\bf 229} (1-3) (2001) 261--292.

\bibitem{leaves-trees} P. J. Slater, Leaves of trees, Proceeding of the 6th Southeastern Conference on Combinatorics, Graph
Theory, and Computing, {\it Congressus Numerantium} {\bf 14} (1975)
549--559.

\bibitem{loc-dom-1} P.J. Slater, Domination and location in acyclic graphs, {\em Networks} {\bf 17} (1987)
55--64.

\bibitem{loc-dom-2} P.J. Slater, Dominating and reference sets in graphs, {\em Journal of Mathematical and Physical Science} {\bf 22}
(1988) 445--455.

\bibitem{roman-1} I. Stewart, Defend the Roman Empire!, {\em Scientific American}, December (1999) 136--138.

\bibitem{vizing1} V. G. Vizing, The Cartesian product of graphs, {\it Vy\v{c}isl. Sistemy} {\bf 9} (1963)
 30--43.

\bibitem{vizing} V. G. Vizing, Some unsolved problems in graph theory, {\it Uspehi Mat. Nauk} {\bf 23} (144) (1968)
 117--134.

\bibitem{corona-sum} K. Williams, On the minimum sum of the corona of two graphs, {\em Congressus
Numerantium} {\bf 94} (1993) 43--49.

\bibitem{domatic-2} B. Zelinka, Domatic number and degrees of vertices of a graph, {\em Mathematica
Slovaca} {\bf 33} (1983) 145--147.

\bibitem{domatic-3} B. Zelinka, Domatic numbers of graphs and their variants: a survey. In:
{\em Domination in graphs, advanced topics} (T.W. Haynes, S.T. Hedetniemi,
and P. Slater, eds., 351-377). Marcel Dekker, New York, (1998).

\bibitem{dim-corona} I. G. Yero, D. Kuziak, J. A. Rodr\'iguez-Vel\'azquez,
On the metric dimension of corona product graphs. \textit{Computers}
\& \textit{Mathematics with Applications} \textbf{61}  (9) (2011)
2793--2798.

\end{thebibliography}
\end{document}